\documentclass[11pt,twoside]{article}
\usepackage{amsmath}
\usepackage{amsfonts}
\usepackage{amsthm,amssymb}
\usepackage{graphicx, color}

\usepackage{indentfirst}

\usepackage[nosort]{cite}
\usepackage{color}
\usepackage{xcolor}%
\usepackage{makeidx}

\newtheorem{theorem}{Theorem}

\newtheorem{lemma}[theorem]{Lemma}
\newtheorem{proposition}[theorem]{Proposition}
\newtheorem{remark}[theorem]{Remark}

\newtheorem{definition}[theorem]{Definition}

\begin{document}
\vskip 1cm
\begin{center}
\vskip 0.1cm {\bf\Large Existence and profile of ground-state solutions to a $1-$Laplacian problem in $\mathbb{R}^N$}
\end{center}
\vskip 0.3cm
\begin{center}

{\sc Claudianor O. Alves$^{1,*}$, Giovany M. Figueiredo$^2$ and  Marcos T. O. Pimenta$^{3}$,}
\\
\vspace{0.5cm}

1. Unidade Acad\^emica de Matem\'atica e Estat\'istica\\ Universidade Federal de Campina Grande \\
58429-900 - Campina Grande - PB , Brazil \\

2.  Departamento de Matem\'atica\\ Universidade de Bras\'ilia-UNB \\
70910-900 - Bras\'ilia - DF , Brazil \\ 

3. Departamento de Matem\'atica e Computa\c{c}\~ao\\ Universidade Estadual Paulista (Unesp), Faculdade de Ci\^encias e Tecnologia\\
19060-900 - Presidente Prudente - SP, Brazil, \\
* corresponding author
\medskip

E-mail addresses: coalves@mat.ufcg.edu.br, giovany@unb.br, pimenta@fct.unesp.br

 \end{center}
\vskip 0.5cm

\begin{abstract}
In this work we prove the existence of ground state solutions for the following class of problems  
\begin{equation*}
\left\{
\begin{array}{ll}
\displaystyle - \Delta_1 u + (1 + \lambda V(x))\frac{u}{|u|} & = f(u),  \quad x \in \mathbb{R}^N, \\
u \in BV(\mathbb{R}^N), &
\end{array} \right.
\label{Pintro}
\end{equation*}
\end{abstract}
where $\lambda > 0$, $\Delta_1$ denotes the $1-$Laplacian operator which is formally defined by $\Delta_1 u = \mbox{div}(\nabla u/|\nabla u|)$, $V:\mathbb{R}^N \to \mathbb{R}$ is a potential satisfying some conditions and $f:\mathbb{R} \to \mathbb{R}$ is a subcritical and superlinear nonlinearity. We prove that for $\lambda > 0$ large enough there exists ground-state solutions and, as $\lambda \to +\infty$, such solutions converges to a ground-state solution of the limit problem in $\Omega = \mbox{int}( V^{-1}(\{0\}))$.

\vskip 1.5cm

 \vskip 0.4cm

\section{Introduction}

\hspace{.5cm} 

Let us consider the following class of quasilinear elliptic problem
$$
\left\{
\begin{array}{ll}
\displaystyle - \Delta_1 u + (1 + \lambda V(x))\frac{u}{|u|} & = f(u),  \quad x \in \mathbb{R}^N, \\
 u \in BV(\mathbb{R}^N), &
\end{array} \right.
\eqno{(P)_\lambda}
$$
where $\lambda > 0$, $N \geq 2$, the operator $\Delta_1$ is the well known $1-$Laplacian operator, whose formal definition is given by $\displaystyle \Delta_1 u = \mbox{div}\left(\frac{\nabla u}{|\nabla u|}\right)$. On the nonlinearity $f$ we assume the following conditions:
\begin{itemize}
\item [$(f_1)$] $f \in C(\mathbb{R})$;
\item [$(f_2)$] $f(s) = o(1)$ as $s \to 0$;
\item [$(f_3)$] There exist constants $c_1, c_2 > 0$ and $p \in (1,1^*)$ such that
$$
|f(s)| \leq c_1 + c_2|s|^{p-1}, \quad \forall s \in \mathbb{R};
$$
\item [$(f_4)$] There exists $\theta > 1 $ such that $$0 < \theta F(s) \leq f(s)s, \quad \mbox{for $s \neq 0$},$$
where $F(s) = \int_0^s f(t)dt$;
\item [$(f_5)$] $f$ is increasing.
\end{itemize}

The potential $V$ is going to be considered satisfying the following conditions:
\begin{itemize}
\item [$(V_1)$] $V(x) \geq 0$, $\forall x \in \mathbb{R}^N$;
\item [$(V_2)$] There exists $M_0 > 0$ such that $|\{x \in \mathbb{R}^N; \, V(x) \leq M_0\}| < +\infty$,
where $|A|$ denotes the Lebesgue measure of a mensurable set $A \subset \mathbb{R}^N$.
\item [$(V_3)$] $\Omega = \mbox{int}(V^{-1}(\{0\})) \neq \emptyset$.
\end{itemize}

Several recent studies have focused on the nonlinear Schr\"{o}dinger equation with deep potential well
\begin{equation} \label{LL2}
- \Delta u + (\lambda a(x)+b(x))u = |u|^{p-2}u \,\,\,  \mbox{ in } \,\,\, \mathbb{R}^N,
\end{equation}
where $a(x), b(x)$ are suitable continuous functions and $ p \in (2, \frac{2N}{N-2} ) $ if $ N \ge 3 $; $ p \in (1, \infty) $ if $ N = 1, 2 $.
In \cite{BW2}, for $b(x)=1$, Bartsch and Wang proved the existence of a least energy solution for $\lambda$ large enough and that the sequence of solutions converges strongly to a least energy solution for a problem in a bounded domain. They also showed the existence of at least cat$\Omega$ positive solutions for large $\lambda$, where $ \Omega = \mbox{int} (a^{-1}(0))$,  and $p$ is close to the critical exponent. In \cite{CD}, Clapp and Ding study the existence of nodal solutions that change sign exactly once, considering the critical growth case. We also refer to \cite{BPW} for nonconstant $b(x)>0$, where the authors prove the existence of $k$ solutions that may change sign for any $k$ and $\lambda$ large enough. For other results related to Schr\"odinger equations with deep potential well, we may refer the readers to \cite{DT, ST, SZ, WZ}.

Motivated by the above references our intention is to prove that some of these results hold for problem $(P)_\lambda$. The main difficulties arise mainly because of the following facts:
\begin{itemize}
\item The lack of smoothness on the energy functional associated to $(P)_\lambda$;
\item The lack of reflexiveness on $BV(\mathbb{R})$, which is the functional space we are going to work with;
\item The difficulty in adapting well known technical results and estimates to our framework, taking into account the way in which we are going to define the sense of solutions.
\end{itemize}

We would like point out that there is in the literature few papers involving the $1$-Laplacian operator in the whole $\mathbb{R}^N$. In fact the authors know only the papers due to Alves and Pimenta \cite{AP} and Figueiredo and Pimenta \cite{FigueiredoPimentaStrauss, FigueiredoPimentaVanishing}. In \cite{AP}, Alves and Pimenta have studied the existence and concentration of solution for the following class of problem 
\begin{equation*}
\left\{
\begin{array}{rr}
\displaystyle - \epsilon \Delta_1 u + V(x)\frac{u}{|u|} & = f(u) \quad \mbox{in $\mathbb{R}^N$,}\\
& u \in BV(\mathbb{R}^N),
\end{array} \right.
\label{Pintro}
\end{equation*}
where $\epsilon>0$ and $V,f$ are continuous functions that satisfy some technical conditions. Actually $f$ has a subcritical growth and $V$ verifies the condition
$$
\liminf_{|z| \rightarrow \infty} V(z) > \inf_{z \in 	\mathbb{R}^N}V(z)=V_{0} >0.
$$ 
In \cite{FigueiredoPimentaStrauss}, Figueiredo and Pimenta has obtained the existence of radially symmetric solutions when $V=1$, by working with the space of radially symmetric $BV$ functions, which is proved to be embedded in $L^q(\mathbb{R}^N)$, for all $q \in (1,1^*)$. In \cite{FigueiredoPimentaVanishing} the same authors shown the existence of ground-state bounded variation solutions for a problem involving the $1-$Laplacian operator and vanishing potentials.

In this work our main result is the following.

\begin{theorem}
Suppose that $f$ satisfies  $(f_1) - (f_5)$ and that $V$ satisfies $(V_1) - (V_3)$, then there exists $\lambda^* > 0$ such that $(P)_\lambda$ has a ground-state bounded variation solution $u^\lambda$ for all $\lambda \geq \lambda^*$. Moreover, there exists $u_\Omega \in BV(\mathbb{R}^N)$ such that, if $\lambda_n \to +\infty$, up to a subsequence, $u_{\lambda_n} \to u_\Omega$ in $L^q_{loc}(\mathbb{R}^N)$, for $1 \leq q < 1^*$ and
$$
\|u_n\|_{\lambda_n} - \|u_\Omega\|_\Omega \to 0, \quad \mbox{as $n \to +\infty$,}
$$
where $u_\Omega \equiv 0$ a.e. in $\mathbb{R}^N\backslash\Omega$ and $u_\Omega$ is a bounded variation solution of 
\begin{equation}
\left\{
\begin{array}{rl}
\displaystyle -\Delta_1 u + \frac{u}{|u|} = f(u) & \mbox{in $\Omega$},\\
u = 0 & \mbox{on $\partial\Omega$.}
\end{array}
\right.
\label{limitproblem}
\end{equation}
\label{maintheorem}
\end{theorem}

Some words about the limit problem (\ref{limitproblem}) are in oder, mainly because the way in which we are going to consider the Dirichlet boundary condition. Note that $u \in BV(\Omega)$ is a bounded variation solution of (\ref{limitproblem}) if
\begin{equation}
\|v\|_\Omega - \|u\|_\Omega \geq \int_{\Omega}f(u)(v - u), \quad \forall v \in BV(\Omega),
\label{BVsolutionlimitproblem}
\end{equation}
where $\|u\|_\Omega = \int_{\Omega}|Du| + \int_{\Omega}|u|dx + \int_{\partial\Omega}|u|d\mathcal{H}_{N-1}$. Since the trace operator from $BV(\Omega)$ into $L^1(\partial\Omega)$ does not have good properties of continuity w.r.t. the $L^q(\Omega)$ convergence, it is completely useless trying to impose the boundary condition in the space, working on $BV_0(\Omega) = \{u \in BV(\Omega); \, \, u = 0 \, \, \mbox{on} \, \, \partial \Omega\}$. In fact our approach follows what is usually done in the literature on $1-$Laplacian problems with Dirichlet boundary conditions, i.e., imposing the boundary conditions by considering the term $\int_{\partial\Omega}|u|d\mathcal{H}_{N-1}$ in the energy functional, where $\mathcal{H}_{N-1}$ denotes the $(N-1)-$dimensional Housdorff measure in $\mathbb{R}^N$.

The paper is organized as follows. In Section 2, we recall some properties involving the space $BV(\mathbb{R}^N)$ and prove some properties of the energy functional associated with the problem. In Section 3, we prove the existence of ground state for $\lambda$ large enough. In the last section we study the concentration arguments and the profile of the solutions as $\lambda \to +\infty$.

\section{Preliminary results}

Let us introduce the space of functions of bounded variation defined by
$$BV(\mathbb{R}^N) = \left\{u \in L^1(\mathbb{R}^N); \, Du \in \mathcal{M}(\mathbb{R}^N,\mathbb{R}^N)\right\}.$$
It can be proved that $u \in BV(\mathbb{R}^N)$ is equivalent to $u \in L^1(\mathbb{R}^N)$ and
$$\int_{\mathbb{R}^N} |Du| := \sup\left\{\int_{\mathbb{R}^N} u \mbox{div}\phi dx; \, \, \phi \in C^1_c(\mathbb{R}^N,\mathbb{R}^N), \, \mbox{s.t.} \, \, |\phi|_\infty \leq 1\right\} < +\infty.$$

The space $BV(\mathbb{R}^N)$ is a Banach space when endowed with the norm
$$\|u\| := \int_{\mathbb{R}^N} |Du| + |u|_1,$$
which is continuously embedded into $L^r(\mathbb{R}^N)$ for all $\displaystyle r \in \left[1,1^*\right]$.

As one can see in \cite{Buttazzo}, the space $BV(\mathbb{R}^N)$ has different convergence and density properties than the usual Sobolev spaces. For example, $C^\infty_0(\mathbb{R}^N)$ is not dense in $BV(\mathbb{R}^N)$ with respect to the strong convergence, since the closure of $C^\infty_0(\mathbb{R}^N)$ in the norm of $BV(\mathbb{R}^N)$ is equal to $W^{1,1}(\mathbb{R}^N)$, which is a proper subspace of $BV(\mathbb{R}^N)$. This has motivated people to define a weaker sense of convergence in $BV(\mathbb{R}^N)$, called {\it intermediate convergence}. We say that $(u_n) \subset BV(\mathbb{R}^N)$ converge to $u \in BV(\mathbb{R}^N)$ in the sense of the intermediate convergence if 
$$
u_n \to u, \quad \mbox{in $L^1(\mathbb{R}^N)$}
$$
and
$$
\int_{\mathbb{R}^N}|Du_n| \to \int_{\mathbb{R}^N}|Du|,
$$
as $n \to \infty$. Fortunately, with respect to the intermediate convergente, $C^\infty_0(\mathbb{R}^N)$ is dense in $BV(\mathbb{R}^N)$. This fact is going to be used later.

For a vectorial Radon measure $\mu \in \mathcal{M}(\mathbb{R}^N,\mathbb{R}^N)$, we denote by $\mu = \mu^a + \mu^s$ the usual decomposition stated in the Radon Nikodyn Theorem, where $\mu^a$ and $\mu^s$ are, respectively, the absolute continuous and the singular parts with respect to the $N-$dimensional Lebesgue measure $\mathcal{L}^N$. We denote by $|\mu|$, the absolute value of $\mu$, the scalar Radon measure defined as in \cite{Buttazzo}[pg. 125]. By $\displaystyle \frac{\mu}{|\mu|}(x)$ we denote the usual Lebesgue derivative of $\mu$ with respect to $|\mu|$, given by
$$\frac{\mu}{|\mu|}(x) = \lim_{r \to 0}\frac{\mu(B_r(x))}{|\mu|(B_r(x))}.$$ 

It can be proved that $\mathcal{J}: BV(\mathbb{R}^N) \to \mathbb{R}$, given by
\begin{equation}
\mathcal{J}(u) = \int_{\mathbb{R}^N} |Du| + \int_{\mathbb{R}^N} |u|dx,
\label{J}
\end{equation}
is a convex functional and Lipschitz continuous in its domain. It is also well know that $\mathcal{J}$ is lower semicontinuous with respect to the $L^r(\mathbb{R}^N)$ topology, for $r \in [1,1^*]$ (see \cite{Giusti} for example). Although non-smooth, the functional $\mathcal{J}$ admits some directional derivatives. More specifically, as is shown in \cite{Anzellotti}, given $u \in BV(\mathbb{R}^N)$, for all $v \in BV(\mathbb{R}^N)$ such that $(Dv)^s$ is absolutely continuous w.r.t. $(Du)^s$ and such that $v$ is equal to $0$ a.e. in the set where $u$ vanishes, it follows that
\begin{equation}
\mathcal{J}'(u)v = \int_{\mathbb{R}^N} \frac{(Du)^a(Dv)^a}{|(Du)^a|}dx + \int_{\mathbb{R}^N} \frac{Du}{|Du|}(x)\frac{Dv}{|Dv|}(x)|(Dv)|^s + \int_{\mathbb{R}^N}\mbox{sgn}(u) v dx,
\label{Jlinha}
\end{equation}
where $\mbox{sgn}(u(x)) = 0$ if $u(x) = 0$ and $\mbox{sgn}(u(x)) = u(x)/|u(x)|$ if $u(x) \neq 0$.
In particular, note that, for all $u \in BV(\mathbb{R}^N)$,
\begin{equation}
\mathcal{J}'(u)u = \mathcal{J}(u).
\label{derivadaJ}
\end{equation}

We have also that $BV(\mathbb{R}^N)$ is a {\it lattice}, i.e., if $u,v \in BV(\mathbb{R}^N)$, then $\max\{u,v\}, \min\{u,v\} \in BV(\mathbb{R}^N)$ and also
\begin{equation}
\mathcal{J}(\max\{u,v\}) + \mathcal{J}(\min\{u,v\}) \leq \mathcal{J}(u) + \mathcal{J}(v), \quad \forall u,v \in BV(\mathbb{R}^N).
\label{lattice}
\end{equation}

Let us denote
$$
E_\lambda = \left\{u \in BV(\mathbb{R}^N); \, \int_{\mathbb{R}^N}(1 + \lambda V(x))|u| dx < +\infty \right\},
$$
the subspace of $BV(\mathbb{R}^N)$ endowed with the following norm
$$
\|v\|_\lambda := \int_{\mathbb{R}^N}|Dv| + \int_{\mathbb{R}^N}(1 + \lambda)V(x)|v|dx.
$$
Note that the embedding $E_\lambda \hookrightarrow BV(\mathbb{R}^N)$ is continuous in such a way that $E_\lambda$ is a Banach space that is continuously embedded into $L^q(\mathbb{R}^N)$, for all $q \in [1,1^*]$.

Let us define the functionals $\Phi_\lambda, \Psi_\lambda, \Phi_F: E_\lambda \to \mathbb{R}$ given by
$$
\Phi_\lambda(u) = \|u\|_\lambda,
$$
$$
\Phi_F(u) = \int_{\mathbb{R}^N}F(u)dx
$$
and
$$
\Psi_\lambda(u) = \Phi_\lambda(u) - \int_{\mathbb{R}^N}F(v)dx.
$$

Note that $\Psi_\lambda$ is written as the difference of a convex locally Lipschitz functional $\Phi_\lambda$, and a $C^1(E)$ one, $\Phi_F$. Then we can use the theory of subdifferentials of Clarke \cite{Clarke} to say that $u^\lambda$ is a critical point of $\Psi_\lambda$ if $0 \in \partial \Psi_\lambda(u^\lambda)$, where $\partial \Psi_\lambda(u^\lambda)$ denotes the subdifferential of $\Psi_\lambda$ in $u^\lambda$. This, in turn, is equivalent to $\Phi_F'(u^\lambda) \in \partial \Phi_\lambda(u^\lambda)$, which is equivalent to
\begin{equation}
\|v\|_\lambda - \|u^\lambda\|_\lambda \geq \int_{\mathbb{R}^N}f(u^\lambda)(v - u^\lambda), \quad \forall v\in E_\lambda.
\label{eqsolution}
\end{equation}

\subsection{The Euler-Lagrange equation}

Since $(P)_\lambda$ contains expressions that doesn't make sense when $\nabla u = 0$ or $u = 0$, then it can be understood just as the formal version of the Euler-Lagrange equation associated to the functional $\Psi_\lambda$. In this section we present the precise form of an Euler-Lagrange equation satisfied by all bounded variation critical points of $\Psi_\lambda$. In order to do so we closely follow the arguments in \cite{Kawohl}, however we have introduced new ideas, because we are working in whole $\mathbb{R}^N$.

The first step is to consider the extension of the functionals $\Phi_\lambda, \Phi_F$ and $\Psi_\lambda$ to $X=L^{1}(\mathbb{R}^N) \cap L^{\frac{N}{N-1}}(\mathbb{R}^N)$ endowed with the norm $\|w\|_{X}=|w|_{1}+|w|_{\frac{N}{N-1}}$, given respectively by $\overline{\Phi_\lambda}, \overline{\Phi_F}, \overline{\Psi}_\lambda: X \to \mathbb{R}\cup \{+\infty\}$, where
$$
\overline{\Phi_\lambda}(v) = 
\left\{
\begin{array}{ll}
\Phi_\lambda(v), & \mbox{if $v \in E_\lambda$},\\
+\infty, & \mbox{if $v \in X\backslash E_\lambda$},
\end{array}
\right.
$$
$$
\overline{\Phi_F}(u) = \int_{\mathbb{R}^N}F(u)dx
$$
and $\overline{\Psi}_\lambda = \overline{\Phi_\lambda} - \overline{\Phi_F}$. It is easy to see that $\overline{\Phi_F}$ belongs to $C^1(X, \mathbb{R})$ and that $\overline{\Phi_\lambda}$ is a convex lower semicontinuous functional defined in $X$. Hence the subdifferential (in the sense of \cite{Szulkin}) of $\overline{\Phi_\lambda}$, denoted by $\partial \overline{\Phi_\lambda}$, is well defined. The following is a crucial result in obtaining an Euler-Lagrange equation satisfied by the critical points of $\Psi_\lambda$.

\begin{lemma}
If  $u^\lambda \in BV(\mathbb{R}^N)$ is such that $0 \in \partial \Psi_\lambda(u^\lambda)$, then $0 \in \partial \overline{\Psi_\lambda}(u^\lambda)$.
\end{lemma}
\begin{proof}
Suppose that $0 \in \partial \Psi_\lambda(u^\lambda)$, i.e., that $u^\lambda$ satisfies (\ref{eqsolution}).
We would like to prove that
$$
\overline{\Phi_\lambda}(v) - \overline{\Phi_\lambda}(u^\lambda) \geq \overline{\Phi_F}\, '(u^\lambda)(v - u^\lambda), \quad \forall v \in X.
$$
To see why, consider $v \in X$ and note that:
\begin{itemize}
\item if $v \in E_\lambda \cap X$, then
\begin{eqnarray*}
\overline{\Phi_\lambda}(v) - \overline{\Phi_\lambda}(u^\lambda) & = & \Phi_\lambda(v) - \Phi_\lambda(u^\lambda)\\
& \geq & \Phi_F'(u^\lambda)(v - u^\lambda)\\
& = & \int_{\mathbb{R}^N}f(u^\lambda)(v - u^\lambda)dx\\
& = & \overline{\Phi_F}\, '(u^\lambda)(v - u^\lambda);
\end{eqnarray*}

\item if $v \in X\backslash E_\lambda$, since $\overline{\Phi_\lambda}(v) = +\infty$ and $\overline{\Phi_\lambda}(u^\lambda) < +\infty$, it follows that
\begin{eqnarray*}
\overline{\Phi_\lambda}(v) - \overline{\Phi_\lambda}(u^\lambda) & = & +\infty\\
& \geq & \overline{\Phi_F}\, '(u^\lambda)(v - u^\lambda).
\end{eqnarray*}
\end{itemize}
Therefore the result follows.
\end{proof}

Let us assume that $u^\lambda \in BV(\mathbb{R}^N)$ is a bounded variation solution of $(P)_\lambda$, i.e., that $u^\lambda$ satisfies (\ref{eqsolution}). Since $0 \in \partial \Psi_\lambda(u^\lambda)$, by the last result it follows that $0 \in \partial \overline{\Psi}_\lambda(u^\lambda)$. Since $\overline{\Phi_\lambda}$ is convex and $\overline{\Phi_F}$ is smooth, it follows that $\overline{\Phi_F}\, '(u^\lambda) \in \partial \overline{\Phi_\lambda}(u^\lambda)$. In what follows, we set 
$\overline{\Phi_\lambda^1}, \overline{\Phi_\lambda^2}:X \to \mathbb{R}\cup \{+\infty\}$ by 
$$
\overline{\Phi_\lambda^1}(v) := 
\left\{
\begin{array}{ll}
\int_{\mathbb{R}^N} |Dv|, & \mbox{if $v \in BV(\mathbb{R}^N)$},\\
+\infty, & \mbox{if $v \in X\backslash BV(\mathbb{R}^N)$},
\end{array}
\right.
$$
and 
$$
\overline{\Phi_\lambda^2}(v) := \int_{\mathbb{R}^N} (1 + \lambda V(x))|v|\,dx.
$$
Note that $\overline{\Phi_\lambda^2} \in C(X,\mathbb{R})$, $\overline{\Phi_\lambda^2} \in C(BV(\mathbb{R}^N),\mathbb{R})$ and 
$$
\overline{\Phi_\lambda}(v)=\overline{\Phi_\lambda^1}(v)+\overline{\Phi_\lambda^2}(v), \quad \forall v \in X.
$$
Since $\overline{\Phi_\lambda^1}$ and $\overline{\Phi_\lambda^2}$ are convex, and $\overline{\Phi_\lambda^2}$ is finite and continuous in every point of $E_\lambda$, it follows from \cite[Theorem 9.5.4]{Buttazzo} that
$$
\overline{\Phi_F}\, '(u^\lambda) \in \partial \overline{\Phi_\lambda}(u^\lambda) = \partial \overline{\Phi_\lambda^1}(u^\lambda) + \partial \overline{\Phi_\lambda^2}(u^\lambda).
$$
By using the same arguments explored in \cite[Theorem 8.15]{Bartle}, it follows that $X' \subset L_{\infty,N}(\mathbb{R}^N)$ where 
$$
L_{\infty,N}(\mathbb{R}^N)=\{g: \mathbb{R}^N \to \mathbb{R} \quad \mbox{measurable}\,:\, ||g||_{\infty,N}<\infty\}
$$ 
where
$$
||g||_{\infty,N}=\sup_{|\phi|_1+|\phi|_{\frac{N}{N-1}}\leq 1}\left|\int_{\mathbb{R}^N}g\phi\,dx\right|.
$$
It is possible to prove that $\|\,\,\|_{\infty,N}$ is a norm in  $L_{\infty,N}(\mathbb{R}^N)$. Moreover, the inclusion $L_{\infty,N}(\mathbb{R}^N) \hookrightarrow L^{N}(B_R(0))$ is continuous for all $R>0$. 

From the above commentaries, there are $z_1^*, z_2^* \in L_{\infty,N}(\mathbb{R}^N)$ such that $z_1^* \in \partial \overline{\Phi_\lambda^1}(u^\lambda)$, $z_2^* \in \partial \overline{\Phi_\lambda}^2(u^\lambda)$ and
$$
\overline{\Phi_F}\, '(u^\lambda) = z_1^* +z_2^* \quad \mbox{ in} \,\,\, L_{\infty,N}(\mathbb{R}^N).
$$
Following the same arguments in \cite[Proposition 4.23, pg. 529]{Kawohl}, we have that there exists $z \in L^\infty(\mathbb{R}^N, \mathbb{R}^N)$ such that $|z|_\infty \leq 1$,
\begin{equation}
-\mbox{div}{z} = z_1^* \quad \mbox{ in} \,\,\, L_{\infty,N}(\mathbb{R}^N)
\label{eulerlagrange1}
\end{equation}
and 
\begin{equation}
 -\int_{\mathbb{R}^N}u^\lambda \mbox{div}z dx = \int_{\mathbb{R}^N}|Du^\lambda|,
 \label{eulerlagrange2}
 \end{equation}
where the divergence in (\ref{eulerlagrange1}) has to be understood in the distributional sense. Moreover, the same result implies that $z_2^*$ is such that 
\begin{equation}
z_2^* |u^\lambda| = (1 + \lambda V(x)) u^\lambda, \quad \mbox{a.e. in $\mathbb{R}^N$.}
\label{eulerlagrange3}
\end{equation}
Therefore, it follows from (\ref{eulerlagrange1}), (\ref{eulerlagrange2}) and (\ref{eulerlagrange3}) that $u^\lambda$ satisfies
\begin{equation}
\left\{
\begin{array}{l}
\exists z \in L^\infty(\mathbb{R}^N,\mathbb{R}^N), \, \, \|z\|_\infty \leq 1,\, \,  \mbox{div}z \in L_{\infty,N}(\mathbb{R}^N), \, \, -\int_{\mathbb{R}^N}u^\lambda \mbox{div}z dx = \int_{\mathbb{R}^N}|Du^\lambda|,\\
\exists z_2^* \in L_{\infty,N}(\mathbb{R}^N),\, \, z_2^*|u^\lambda| = (1 + \lambda V(x))u^\lambda \quad \mbox{a.e. in $\mathbb{R}^N$},\\
-\mbox{div} z + z_2^* = f(u^\lambda), \quad \mbox{a.e. in $\mathbb{R}^N$}.
\end{array}
\right.
\label{eulerlagrangeequation}
\end{equation}

Hence, (\ref{eulerlagrangeequation}) is the precise version of $(P)_\lambda$.

\section{Existence of solution}
\label{sectionexistence}

Let us first verify that the geometrical conditions of the Mountain Pass Theorem are satisfied by $\Psi_\lambda$.
\begin{lemma}
There exist $\alpha, \rho > 0$ (uniform in $\lambda$) such that,
\begin{itemize}
\item [$i)$] $\Psi_\lambda(u) \geq \alpha$ for all $u \in E_\lambda$ such that $\|u\|_\lambda = \rho$, for all $\lambda > 0$;
\item [$ii)$] For each $\lambda > 0$, there exists $e_\lambda \in E_\lambda$ such that $\|e_\lambda\|_\lambda > \rho$ and $\Psi_\lambda(e_\lambda) < 0$.
\end{itemize}
\label{mountainpass}
\end{lemma}
\begin{proof}
By $(f_2)-(f_3)$, it follows that for each $\eta > 0$, there exists $A_\eta > 0$ such that
\begin{equation}
|F(s)| \leq \eta|s| + A_\eta|s|^p, \quad \forall s \in \mathbb{R}.
\label{fepsilon}
\end{equation}
Note that, by (\ref{fepsilon}) and the embeddings of $E_\lambda$, 
\begin{eqnarray*}
\Psi_\lambda(u)& = &\int_{\mathbb{R}^N} |Du| + \int_{\mathbb{R}^N}(1 + \lambda V(x)) |u| dx - \int_{\mathbb{R}^N} F(u)dx\\
& \geq & \|u\|_\lambda - \eta |u|_1 - A_\eta|u|_p^p\\
& \geq & \|u\|_\lambda - \eta \|u\|_\lambda - c_3\|u\|_\lambda^p\\
& = & \|u\|_\lambda\left(1 - \eta - c_3\|u\|_\lambda^{p-1}\right)\\
& \geq & \alpha,
\end{eqnarray*}
for all $u \in E_\lambda$, such that $\|u\|_\lambda = \rho$, where $0 < \eta < 1$ is fixed, $\displaystyle 0 < \rho < \left(\frac{1-\eta}{c_3}\right)^\frac{1}{p-1}$ and $\displaystyle \alpha = \rho(1 - \eta - c_3\rho^{p-1})$.

In order to verify $ii)$ note that by $(f_4)$, there exist constants $d_1,d_2 > 0$ such that
\begin{equation}
F(s) \geq d_1|s|^\theta - d_2 ,\quad \forall s \in \mathbb{R}.
\label{Festimate}
\end{equation}
If $u$ is a function in $E_\lambda \backslash\{0\}$ with compact support, we derive that 
\begin{equation}
\Psi_\lambda(tu) \leq t\|u\|_\lambda - d_1t^\theta |u|_\theta^\theta + d_2|\mbox{supp}(u)| \to -\infty,
\label{MPTsecondestimate}
\end{equation}
as $t \to +\infty$. Since $\theta > 1$, we can choose $e_\lambda \in E_\lambda$ such that $\Psi(e_\lambda) < 0$.
\end{proof}

By \cite[Theorem 1.3]{FigueiredoPimentaStrauss} it follows that, for all $\lambda > 0$, there exists a sequence $(u^\lambda_n) \subset E_\lambda$ such that
\begin{equation}
\Psi_\lambda(u^\lambda_n) = c_\lambda + o_n(1)
\label{MP1}
\end{equation}
and
\begin{equation}
\|v\|_\lambda - \|u_n^\lambda\|_\lambda \geq \int_{\mathbb{R}^N}f(u_n^\lambda)(v - u^\lambda_n) - \tau_n\|v - u^\lambda_n\|_\lambda, \quad \forall v \in E_\lambda,
\label{MP2}
\end{equation}
where $\tau_n \to 0$ as $n \to +\infty$. The minimax value $c_\lambda$ is given by
$$
c_\lambda = \inf_{\gamma \in \Gamma_\lambda} \max_{t \in [0,1]}\Psi_\lambda (\gamma(t)),
$$
where $\Gamma_\lambda = \{\gamma \in C([0,1], E_\lambda); \, \gamma(0) = 0 \, \, \mbox{and} \, \, \Psi_\lambda(\gamma(1)) < 0\}$.
 Note that by Lemma \ref{mountainpass},
\begin{equation}
c_\lambda \geq \alpha > 0, \quad \forall \lambda > 0.
\label{estimateclambda}
\end{equation}

In our approach will be important the so called {\it Nehari set}, defined as
$$
\begin{array}{ll}
\displaystyle \mathcal{N}_\lambda &\displaystyle  = \{u \in E_\lambda \backslash\{0\}; \, \Psi_\lambda'(u)u = 0\}\\
\displaystyle & \displaystyle = \left\{u \in E_\lambda \backslash\{0\}; \, \|u\|_\lambda = \int_{\mathbb{R}^N}f(u)u \, dx \right\}.
\end{array}
$$
This set is going to give us a better characterization of the minimax level $c_\lambda$. From (\ref{Jlinha}), $\mathcal{N}_\lambda$ is a set that contains all nontrivial bounded variation solutions of $(P)_\lambda$.  Its definition is based on arguments that can be found in \cite{FigueiredoPimentaNehari} which, in turn, are strongly influenced by those ones in \cite{Rabinowitz}. More specifically, they consist in performing a study of the {\it fibering maps} $\gamma_u(t):= \Psi_\lambda(tu)$, by using $(f_1) - (f_5)$ to show that $\mathcal{N}_\lambda$ is radially homeomorphic to the unit sphere in $E_\lambda$. In fact, for each $u \in E_\lambda\backslash\{0\}$, by $(f_2)$ and $(f_3)$, it can be seen that there exists $t_0 > 0$ such that $\gamma_u(t_0) > 0$. On the other hand, $(f_4)$ implies that $\gamma_u(t) \to -\infty$ as $t \to +\infty$. Then there exists $t_u > 0$ such that $\gamma_u(t_u) = \max_{t > 0} \gamma_u(t)$ and then that $\gamma_u'(t_u) = 0$. But $(f_5)$ implies that such $t_u$ is unique. Then for each $u \in E_V\backslash\{0\}$, there exists a unique $t_u > 0$ such that $t_u u \in \mathcal{N}_\lambda $. This establishes such a radial homeomorphism. Still with arguments presented in \cite{Rabinowitz}, one can prove that the minimax level $c_\lambda$ satisfies
\begin{equation}
c_\lambda = \inf_{u \in E_\lambda \backslash\{0\}}\max_{t \geq 0} \Psi_\lambda(tu) = \inf_{u \in \mathcal{N}_\lambda}\Psi_\lambda(u).
\label{c_lambdaNehari}
\end{equation}

\begin{lemma}\label{estimatec_lambda}
There exist constants $\alpha_0,\alpha_1  > 0$ which do not depend on $\lambda > 0$, such that
$$
\alpha_0 \leq c_\lambda \leq \alpha_1, \quad \forall \lambda > 0.
$$
\end{lemma}
\begin{proof}
By Lemma \ref{mountainpass} it is enough to take $\alpha_0 \in (0, \alpha)$. In order to obtain $\alpha_1$, let us fix $\varphi \in C^\infty_0(\Omega)$. Then, for all $t > 0$, as in (\ref{MPTsecondestimate}) we get
$$
\Psi_\lambda(t\varphi) \leq t\left(\int_{\mathbb{R}^N}|D\varphi| + \int_{\mathbb{R}^N}|\varphi| dx \right) - d_1t^\theta |\varphi|_\theta^\theta + d_2|\mbox{supp}(\varphi)| \to -\infty,
$$
as $t \to +\infty$.
Hence if $\alpha_1 =: \max_{t > 0}\Psi_\lambda(t\varphi) > 0$, it follows from the definition of $c_\lambda$ that 
$$
c_\lambda \leq \alpha_1, \quad \forall \lambda > 0.
$$
\end{proof}
Now let us study some more refined information about the sequence $(u_n)_{n \in \mathbb{N}}$.

\begin{lemma}
The sequence $(u^\lambda_n)_{n \in \mathbb{N}}$ is bounded in $E_\lambda$.
\label{unlambdabounded}
\end{lemma}
\begin{proof}
Considering $v = 2u_n^\lambda$ in (\ref{MP2}), we obtain
$$
\|u_n^\lambda\|_\lambda \geq \int_{\mathbb{R}^N}f(u_n^\lambda)u_n^\lambda dx - \tau_n\|u_n^\lambda\|_\lambda,
$$
or equivalently
\begin{equation}
(1 + \tau_n)\|u_n^\lambda\|_\lambda \geq \int_{\mathbb{R}^N}f(u_n^\lambda)u_n^\lambda dx.
\label{lemmaboundedeq1}
\end{equation}
Then, by $(f_4)$ and (\ref{lemmaboundedeq1}), 
\begin{eqnarray*}
c_\lambda + o_n(1) & \geq & \Psi_\lambda(u_n^\lambda)\\
& = & \|u_n^\lambda\|_\lambda + \int_{\mathbb{R}^N}\left(\frac{1}{\theta}f(u_n^\lambda)u_n^\lambda - F(u_n^\lambda)\right)dx - \int_{\mathbb{R}^N}\frac{1}{\theta}f(u_n^\lambda)u_n^\lambda dx\\
& \geq & \|u_n^\lambda\|_\lambda \left(1- \frac{1}{\theta} - \frac{\tau_n}{\theta}\right)\\
& \geq & C\|u_n^\lambda\|_\lambda,
\end{eqnarray*}
for some $C > 0$ that does not depend on $n \in \mathbb{N}$ nor $\lambda > 0$.
\end{proof}
\begin{remark}\label{remarkbounded}
Note that by Lemmas \ref{estimatec_lambda} and \ref{unlambdabounded}, there exists a constant $C > 0$ that does not depend on $\lambda$, such that
$$
\|u_n^\lambda\|_\lambda \leq C, \quad \forall n \in \mathbb{N}.
$$
\end{remark}

By Lemma \ref{unlambdabounded} and the compactness of the embeddings of $BV(K)$ in $L^q(K)$ for $1 \leq q < 1^*$ and $K \subset \mathbb{R}^N$ compact, there exists $u^\lambda \in BV_{loc}(\mathbb{R}^N)$ such that 
\begin{equation}
u_n^\lambda \to u^\lambda \quad \mbox{in $L^q_{loc}(\mathbb{R}^N)$ for $1 \leq q < 1^*$}
\label{convergenceun}
\end{equation}
and
\begin{equation} \label{limite pontual}
u_n^\lambda(x) \to u^\lambda(x) \quad \mbox{a.e. $x \in \mathbb{R}^N$,}
\end{equation}
as $n \to +\infty$. Moreover $u^\lambda$ belongs to $BV(\mathbb{R}^N)$ and then to $E_\lambda$ (by using Fatou Lemma and the boundedness of the sequence $(\|u_n\|_{\lambda})_{n \in \mathbb{N}}$). In fact, if $R > 0$, by the semicontinuity of the norm in $BV(B_R(0))$ w.r.t. the $L^1(B_R(0))$ topology it follows that
\begin{equation}
\|u^\lambda\|_{BV(B_R(0))} \leq \liminf_{n \to +\infty}\|u_n^\lambda\|_{BV(B_R(0))} \leq \liminf_{n \to +\infty}\|u_n^\lambda\|_{BV(\mathbb{R}^N)} \leq C,
\label{ulambdaBV}
\end{equation}
where $C$ does not depend on $n$ nor on $R$.
Since the last inequality holds for every $R > 0$, then $u^\lambda \in BV(\mathbb{R}^N)$.

The next result will help us to get some compactness properties involving the sequence $(u^\lambda_n)$.

\begin{lemma}
Fix $q \in [1,1^*)$. Then, for a given $\epsilon > 0$, there exists $\lambda^* > 0$ and $R > 0$ such that
\begin{equation}
\int_{\mathbb{R}^N\backslash B_R(0)} |u_n^\lambda|^q dx \leq \epsilon,
\label{estimateexterior}
\end{equation}
for all $\lambda \geq \lambda^*$ and $n \in \mathbb{N}$.
\label{lemmaestimateBRc}
\end{lemma}
\begin{proof}
In fact, for a given $R > 0$, let us define the sets
$$
A(R) = \{x \in \mathbb{R}^N; \, |x| > R \, \, \mbox{and} \, \, V(x) \geq M_0\}
$$
and
$$
B(R) = \{x \in \mathbb{R}^N; \, |x| > R \, \, \mbox{and} \, \, V(x) < M_0\},
$$
where $M_0$ is given in $(V_2)$.

Note that, by Remark \ref{remarkbounded} and $(V_2)$, 
\begin{equation}\label{A(R)}
\int_{A(R)}|u_n^\lambda| dx \leq  \frac{1}{\lambda M_0 + 1}\|u_n\|_\lambda \leq \frac{C}{\lambda M_0 + 1} < \frac{\epsilon}{2}, \quad \forall n \in \mathbb{N},
\end{equation}
if $\lambda > \lambda^*$ where $\displaystyle \lambda^* \geq M_0^{-1}\left(\frac{2C}{\epsilon} - 1\right)$.

On the other hand, again by Remark \ref{remarkbounded}, H\"older inequality and the embeddings of $E_\lambda$, 
\begin{equation}\label{B(R)}
\int_{B(R)}|u_n^\lambda| dx  \leq  C|u_n^\lambda|_{1^*}^{1^*}|B(R)|^\frac{1}{N} \leq C|B(R)|^\frac{1}{N} < \frac{\epsilon}{2}
\end{equation}
if $R > 0$ is large enough, since by $(V_2)$, $|B(R)| \to 0$ as $R \to +\infty$.

Then, if $\lambda > \lambda ^*$ and $R > 0$ is large enough, from (\ref{A(R)}) and (\ref{B(R)}) it follows the claim for $q = 1$. Now by Remark \ref{remarkbounded}, the estimate for $q \in (1,1^*)$ follows from interpolation in Lebesgue spaces since $(u_n^\lambda)$ is bounded (uniformly in $\lambda$) in $L^{1^*}(\mathbb{R}^N)$.

\end{proof}

The next result will be used to show that $u^\lambda \neq 0$.

\begin{lemma}\label{c_lambdabelow}
\begin{equation} \label{NOVAEQ}
\liminf_{n \to +\infty}\|u_n^\lambda\|_\lambda \geq \alpha_0\, \quad \forall \lambda > 0.
\end{equation}
\end{lemma}
\begin{proof}
Note that from (\ref{MP1}) and Lemma \ref{estimatec_lambda},
$$
\alpha_0 + o_n(1) \leq c_\lambda + o_n(1) = \Psi_\lambda(u_n^\lambda) = \|u_n^\lambda\|_\lambda - \int_{\mathbb{R}^N}F(u_n^\lambda)dx \leq \|u_n^\lambda\|_\lambda.
$$
\end{proof}

\begin{lemma}
For $\lambda^*$ as in Lemma \ref{lemmaestimateBRc}, it follows that $u^\lambda \neq 0$ for all $\lambda \geq \lambda^*$.
\label{lemmaulambdanaonula}
\end{lemma}

\begin{proof}
Considering in (\ref{MP2}) $v = u_n^\lambda + tu_n^\lambda$ and taking the limit as $t \to 0^\pm$, we find
$$
\Psi_\lambda'(u_n^\lambda)u_n^\lambda = o_n(1),
$$
which implies that
\begin{eqnarray}
\nonumber \|u_n^\lambda\|_\lambda & = & \int_{\mathbb{R}^N}f(u_n^\lambda)u_n^\lambda dx + o_n(1)\\
& = & \int_{B_R(0)}f(u_n^\lambda)u_n^\lambda dx + \int_{\mathbb{R}^N \backslash B_R(0)}f(u_n^\lambda)u_n^\lambda dx + o_n(1).
\label{unlambdaf}
\end{eqnarray}
From $(f_3)$, 
\begin{equation}
\int_{\mathbb{R}^N \backslash B_R(0)}f(u_n^\lambda)u_n^\lambda dx \leq c_1\int_{\mathbb{R}^N \backslash B_R(0)}|u_n^\lambda| dx + c_2\int_{\mathbb{R}^N \backslash B_R(0)}|u_n^\lambda|^p dx.
\label{estimatefBRc}
\end{equation}
By taking  $q=p$ and $\epsilon$ small enough in Lemma \ref{lemmaestimateBRc}, the inequality (\ref{estimatefBRc}) gives that
\begin{equation}
\limsup_{n \to +\infty}\int_{\mathbb{R}^N \backslash B_R(0)}f(u_n^\lambda)u_n^\lambda dx \leq \frac{\alpha_0}{2},
\label{estimatefBRc2}
\end{equation}
where $\alpha_0$ is as in Lemma \ref{c_lambdabelow}.

From the compactness of the embeddings $BV(B_R(0)) \hookrightarrow L^q(B_R(0))$ for $q \in [1,1^*)$ and $(f_3)$, we have that
\begin{equation}
\lim_{n \to +\infty} \int_{B_R(0)}f(u_n^\lambda)u_n^\lambda dx = \int_{B_R(0)}f(u^\lambda)u^\lambda dx.
\label{estimatefBRc3}
\end{equation}
Hence, from (\ref{NOVAEQ}), (\ref{unlambdaf}), (\ref{estimatefBRc2}) and (\ref{estimatefBRc3}), 
\begin{eqnarray*}
\int_{B_R(0)}f(u^\lambda)u^\lambda dx & = & \lim_{n \to +\infty}\int_{B_R(0)}f(u_n^\lambda)u_n^\lambda dx\\
& \geq &  \liminf_{n \to +\infty}\left( \|u_n^\lambda\|_\lambda - \int_{\mathbb{R}^N \backslash B_R(0)}f(u_n^\lambda)u_n^\lambda dx \right)\\
& \geq & \frac{\alpha_0}{2},
\end{eqnarray*}
if $\lambda \geq \lambda^*$. This implies that $u^\lambda \neq 0$.
\end{proof}

\begin{lemma}
$\displaystyle \Phi_\lambda'(u^\lambda)u^\lambda \leq 0$.
\label{lemmaNehari}
\end{lemma}
\begin{proof}
Note that, if $\varphi \in C^\infty_0(\mathbb{R}^N)$, $0 \leq \varphi \leq 1$, $\varphi \equiv 1$ in $B_1(0)$ and $\varphi \equiv 0$ in $B_{2}(0)^c$, for $\varphi_R := \varphi (\cdot/R)$, it follows that for all $u \in BV(\mathbb{R}^N)$, 
\begin{equation}
(D(\varphi_R u))^s \quad \mbox{is absolutely continuous w.r.t.} \quad (Du)^s.
\label{derivadafuncaoteste}
\end{equation} 
In fact, note that
$$
D(\varphi_R u) = \nabla \varphi_R u + \varphi_R Du = \nabla \varphi_R u + \varphi_R Du^a + \varphi_R Du^s , \quad \mbox{in $\mathcal{D}'(\mathbb{R}^N)$}.
$$
Then it follows that 
$$
(D(\varphi_R u))^s = (\varphi_R (Du)^s)^s = \varphi_R (Du)^s.
$$

Taking (\ref{derivadafuncaoteste}) into account, the fact that $\varphi_R u_n^\lambda$ is equal to $0$ a.e. in the set where $u_n^\lambda$ vanishes and also the fact that $\frac{\varphi_R \mu}{|\varphi_R \mu|} = \frac{\mu}{|\mu|}$ a.e. in $B_R(0)$, it follows that it is well defined $\Psi_\lambda'(u_n^\lambda)(\varphi_R u_n^\lambda)$. Moreover, by (\ref{Jlinha}), it follows that
\begin{eqnarray*}
\Psi_\lambda'(u_n^\lambda)(\varphi_R u_n^\lambda) & = & \int_{\mathbb{R}^N}\frac{((Du_n^\lambda)^a)^2 \varphi_R + u_n^\lambda(Du_n^\lambda)^a \cdot \nabla \varphi_R}{|(Du_n^\lambda)^a|}dx \\
& & + \int_{\mathbb{R}^N}\frac{Du_n^\lambda}{|Du_n^\lambda|}\frac{\varphi_R(Du_n^\lambda)^s}{|\varphi_R (Du_n^\lambda)^s|}|\varphi_R(D u_n^\lambda)^s| + \\
& & + \int_{\mathbb{R}^N}(1 + \lambda V(x)) \mbox{sgn}(u_n^\lambda)(\varphi_R u_n^\lambda)dx - \int_{\mathbb{R}^N}f(u_n^\lambda)\varphi_R u_n^\lambda dx\\
& = & \int_{\mathbb{R}^N}\varphi_R |(Du_n^\lambda)^a|dx + \int_{\mathbb{R}^N}\frac{u_n^\lambda(Du_n^\lambda)^a \cdot \nabla\varphi_R}{|(Du_n^\lambda)^a|}dx +\\
& & + \int_{\mathbb{R}^N}\frac{(Du_n^\lambda)^s}{|(Du_n^\lambda)^s|}\frac{\varphi_R(Du_n^\lambda)^s}{|\varphi_R (Du_n^\lambda)^s|}|\varphi_R(D u_n^\lambda)^s| + \int_{\mathbb{R}^N}(1 + \lambda V(x)) |u_n^\lambda|\varphi_R dx -\\
& & - \int_{\mathbb{R}^N}f(u_n^\lambda)\varphi_R u_n^\lambda dx.
\end{eqnarray*}
The last equality together with the lower semicontinuity of the norm in $BV(B_R(0))$ w.r.t. the $L^1(B_R(0))$ convergence and the fact that $\Psi_\lambda'(u_n^\lambda)(\varphi_R u_n^\lambda) = o_n(1)$ (since $(\varphi_R u_n^\lambda)$ is bounded in $BV(\mathbb{R}^N))$, imply that
\begin{equation}
\int_{B_R(0)}|Du^\lambda| + \liminf_{n \to \infty}\int_{\mathbb{R}^N}\frac{u_n^\lambda(Du_n^\lambda)^a \cdot \nabla\varphi_R}{|(Du_n^\lambda)^a|}dx + \int_{\mathbb{R}^N}(1 + \lambda V(x))\varphi_R|u^\lambda|dx \leq \int_{\mathbb{R}^N}f(u^\lambda)u^\lambda \varphi_Rdx.
\label{eqv}
\end{equation}
By doing $R \to +\infty$ in both sides of (\ref{eqv}) we get that
\begin{equation}
\int_{\mathbb{R}^N}|Du^\lambda| + \int_{\mathbb{R}^N}(1 + \lambda V(x))|u^\lambda|dx \leq \int_{\mathbb{R}^N}f(u^\lambda)u^\lambda dx,
\label{inequalityv}
\end{equation}
and the proof is finished. 
\end{proof}

By the last result there exists $t_\lambda \in (0,1]$ such that $t_\lambda u^\lambda \in \mathcal{N}_\lambda$.

Before to get more information about $t_\lambda$, let's just give a piece of information.

\begin{lemma}
\label{lemmaincreasing}
Under $(f_5)$, $f$ is such that $t \mapsto f(t)t - F(t)$ is increasing for $t \in (0,+\infty)$ and decreasing for $t \in (-\infty, 0)$.
\end{lemma}
\begin{proof}
Let $t_1 > t_2 > 0$, then
\begin{eqnarray*}
f(t_{1})t_{1} - F(t_{1}) & = &f(t_{1})t_{1} - F(t_{2}) - \int_{t_{2}}^{t_{1}}f(s)ds\\
& > &f(t_{1})t_{1} - F(t_{2}) - f(t_{1})\int_{t_{2}}^{t_{1}}ds\\
& > &f(t_{2})t_{2} - F(t_{2}).
\end{eqnarray*}
The case in which $t_1 < t_2 < 0$ is analogous.
\end{proof}

\begin{lemma}\label{lemmatlambda}
Let $\lambda^*$ be as in Lemma \ref{lemmaestimateBRc}. If $\lambda \geq \lambda^*$, then $t_\lambda = 1$,
$$
\lim_{n \to +\infty}\int_{\mathbb{R}^N}f(u_n^\lambda)u_n^\lambda dx = \int_{\mathbb{R}^N}f(u^\lambda)u^\lambda dx,
$$ 
$$
\lim_{n \to +\infty}\int_{\mathbb{R}^N}F(u_n^\lambda) dx = \int_{\mathbb{R}^N}F(u^\lambda) dx
$$
and
$$
\lim_{n \to +\infty}\|u_n^\lambda\|_\lambda = \|u^\lambda\|_\lambda.
$$
\end{lemma}
\begin{proof}
Note that
\begin{equation}
c_\lambda + o_n(1) = \Psi_\lambda(u^{\lambda}_n) +o_n(1) = \Psi_\lambda(u^{\lambda}_n) - \Psi'_\lambda(u^{\lambda}_n)u^{\lambda}_n = \int_{\mathbb{R}^N}\left(f(u^{\lambda}_n)u^{\lambda}_n - F(u^{\lambda}_n)\right) dx.
\label{convergenciaf}
\end{equation}
Applying Fatou Lemma in the last inequality together with Lemma \ref{lemmaincreasing}, we derive that
\begin{eqnarray*}
c_\lambda & \geq &  \int_{\mathbb{R}^N}\left(f(u^\lambda)u^\lambda - F(u^\lambda)\right) dx\\
& \geq & \int_{\mathbb{R}^N}\left(f(t_\lambda u^\lambda)t_\lambda u^\lambda - F(t_\lambda u^\lambda)\right) dx\\
& = & \Psi_\lambda(t_\lambda u^\lambda) - \Psi'_\lambda(t_\lambda u^\lambda)t_\lambda u^\lambda\\
& = & \Psi_\lambda(t_\lambda u^\lambda)\\
& \geq & c_\lambda.
\end{eqnarray*}
Hence, $t_\lambda=1$, $\Phi_{\lambda}(u^\lambda)=c_\lambda$, and by (\ref{convergenciaf}),
\begin{equation} \label{EQUAT0}
f(u_n^\lambda)u_n^\lambda - F(u_n^\lambda) \to f(u^\lambda)u^\lambda - F(u^\lambda) \quad \mbox{in} \quad L^{1}(\mathbb{R}^N).
\end{equation}
This limit together with $(f_4)$ and (\ref{limite pontual}) yield
\begin{equation} \label{EQUAT1}
f(u_n^\lambda)u_n^\lambda \to f(u^\lambda)u^\lambda \quad \mbox{in $L^1(\mathbb{R}^N)$}
\end{equation}
\begin{equation} \label{EQUAT2}
F(u_n^\lambda) \to F(u^\lambda) \quad \mbox{in $L^1(\mathbb{R}^N)$}
\end{equation}
and
\begin{equation} \label{EQUAT3}
\|u_n^\lambda\|_\lambda \to \|u^\lambda\|_\lambda.
\end{equation}

Here, we have used the fact that $(f_4)$ ensures that 
$$
0\leq (1- 1/ \theta)f(u_n^\lambda)v_n \leq f(u_n^\lambda)v_n - F(u_n^\lambda)
$$
and
$$
0 \leq (\theta- 1)F(u_n^\lambda) \leq f(u_n^\lambda)u_n^\lambda - F(u_n^\lambda).
$$
Then, by (\ref{EQUAT0}), we can apply the Lebesgue Dominated Convergence Theorem to get (\ref{EQUAT1}) and (\ref{EQUAT2}). Recalling that $\|u^\lambda\|_{\lambda}=\int_{\mathbb{R}^N}f(u^\lambda)u^\lambda$ and $\|u_n^\lambda\|_{\lambda}=\int_{\mathbb{R}^N}f(u_n^\lambda)u_n^\lambda + o_n(1)$, the limit (\ref{EQUAT1}) implies in (\ref{EQUAT3}).
\end{proof}

As a consequence of the last result, we see that $u^\lambda$ is a bounded variation solution of $(P)_\lambda$. In fact, from (\ref{MP2}), Lemma \ref{lemmatlambda} and the lower semicontinuity of $\|\cdot \|_\lambda$ w.r.t. the $L^1(\mathbb{R}^N)$ convergence, it follows that
\begin{equation}
\|v\|_\lambda - \|u^\lambda\|_\lambda \geq \int_{\mathbb{R}^N}f(u^\lambda)(v - u^\lambda)dx, \quad \forall v \in E_\lambda,
\label{solutionepsilon}
\end{equation}
and then $u^\lambda$ is in fact a nontrivial solution of $(P)_\lambda$. Moreover, note that from (\ref{MP1})
\begin{eqnarray*}
c_\lambda & \leq & \Psi_\lambda (u^\lambda)\\
& = & \Psi_\lambda (u^\lambda) - \Psi_\lambda'(u^\lambda)u^\lambda\\
& = & \int_{\mathbb{R}^N}\left(f(u^\lambda)u^\lambda - F(u^\lambda)\right)dx\\
& \leq & \liminf_{n \to \infty}\int_{\mathbb{R}^N}\left(f(u_n^\lambda)u_n^\lambda - F(u_n^\lambda)\right)dx\\
& = & \Psi_\lambda(u_n^\lambda) + o_n(1)\\
& = & c_\lambda,
\end{eqnarray*}
which implies that 
\begin{equation}
\Psi_\lambda(u^\lambda) = c_\lambda.
\label{ulambdaclambda}
\end{equation}

Since $\mathcal{N}_\lambda$ contains all nontrivial bounded variation solutions of $(P)_\lambda$, from (\ref{ulambdaclambda}), in view of (\ref{c_lambdaNehari}) it follows that $u^\lambda$ is a ground-state solution of $(P)_\lambda$.

\section{The concentration arguments}

\subsection{The behavior of the $(PS)_{c,\infty}$ sequences}

First of all let us consider the following definition.

\begin{definition}
A sequence $(u_n) \subset  BV(\mathbb{R}^N)$ is called a $ (PS)_{c,\infty}$-sequence for the family $\left(\Psi_\lambda \right)_{\lambda \geq 1}$, if there is a sequence $\lambda_n \to \infty $ such that $u_n \in E_{\lambda_n}$ for $n \in \mathbb{N}$, 
$$
   \Psi_{\lambda_n }(u_n) \to c,
$$
as $n \to +\infty$ and moreover
\begin{equation}
\|v\|_{\lambda_n} - \|u_n\|_{\lambda_n} \geq \int_{\mathbb{R}^N}f(u_n)(v - u_n) - \tau_n\|v - u_n\|_{\lambda_n}, \quad \forall v \in E_{\lambda_n}
\label{PScinfinity}
\end{equation}
where $\tau_n \to 0$ as $n \to +\infty$.
\end{definition}

Before to proceed with other results, let us point out some facts about the limit problem (\ref{limitproblem}). Note that (\ref{limitproblem}) is the formal version of the Euler-Lagrange equation of the functional $\Phi_\Omega : BV(\Omega) \to \mathbb{R}$ given by
$$
\Phi_\Omega(u) = \int_\Omega |Du| + \int_\Omega |u|dx + \int_{\partial\Omega} |u|d\mathcal{H}_{N-1} - \int_\Omega F(u) dx.
$$

By the assumptions $(f_1)-(f_5)$, the {\it Nehari set} associated to $\Phi_\Omega$ is well defined by
$$
\mathcal{N}_\Omega = \{u \in BV(\Omega)\backslash\{0\}; \, \Phi_\Omega'(u)u = 0\},
$$
and let us define 
$$
c_\Omega = \inf_{\mathcal{N}_\Omega}\Phi_\Omega.
$$
Note that $c_\Omega$ is well defined since from $(f_4)$, $\Phi_\Omega(u) \geq 0$ for all $u \in \mathcal{N}_\Omega$. In fact, since $I_\Omega$ satisfies the geometric conditions of the Mountain Pass Theorem, well known arguments imply that 
$$
c_\Omega = \inf_{\gamma \in \Gamma_\Omega} \max_{t \in [0,1]}I_\Omega(\gamma(t)),
$$
where $\Gamma_\Omega = \{\gamma \in C([0,1],BV(\Omega)); \,\, \gamma(0) = 0 \, \, \mbox{and} \, \, I_\Omega(\gamma(1)) < 0\}$.

\begin{proposition} \label{(PS)inftycondition}
Let $(u_n) \subset BV(\mathbb{R}^N)$ be a $ (PS)_{d,\infty}$-sequence for $ \left(\Psi_\lambda \right)_{\lambda \geq 1}$, where $d \in \mathbb{R}$. Then either $d = 0$ or $d \geq c_\Omega$. In the last case, there exists $u_\Omega \in BV(\mathbb{R}^N)$ such that, up to a subsequence, $u_n \to u_\Omega$ in $L^q_{loc}(\mathbb{R}^N)$, for $1 \leq q < 1^*$, $u_\Omega \equiv 0$ a.e. in $\mathbb{R}^N\backslash\Omega$ and $u_\Omega$ is a bounded variation solution of (\ref{limitproblem}). Moreover, if $d = c_\Omega$, then 
$$
\|u_n\|_{\lambda_n} - \|u_\Omega\|_\Omega \to 0, \quad \mbox{as $n \to +\infty$.}
$$
\end{proposition}
\begin{proof}
First of all note that the arguments in Lemma \ref{unlambdabounded} imply that $d \geq 0$, since it holds
\begin{equation}
d + o_n(1) \geq C\|u_n\|_{\lambda_n}.
\label{dun}
\end{equation}

It follows also from (\ref{dun}) that $(\|u_n\|_{\lambda_n})$ is a bounded sequence in $\mathbb{R}$ and then that $(u_n)$ is bounded in $BV(\mathbb{R}^N)$. By the Sobolev embeddings, there exists $u_\Omega \in BV_{loc}(\mathbb{R}^N)$ such that $u_n \to u_\Omega$ in $L^q_{loc}(\mathbb{R}^N)$, for $1 \leq q < 1^*$. Moreover, it is possible to argue as in the last section in order to show that in fact $u_\Omega \in BV(\mathbb{R}^N)$.

Now let us show that $u_\Omega \equiv 0$ a.e. in $\mathbb{R}^N\backslash\Omega$. For each $m \in \mathbb{N}$, let us define $C_m = \{x \in \mathbb{R}^N; \, V(x) \geq 1/m\}$. Then note that
\begin{eqnarray*}
\int_{C_m}|u_n| dx & \leq & \frac{m}{\lambda_n}\int_{C_m}\lambda_n V(x)|u_n|dx\\
& \leq & \frac{m}{\lambda_n}\|u_n\|_{\lambda_n}\\
& = & o_n(1),
\end{eqnarray*}
since $(\|u_n\|_{\lambda_n})$ is a bounded sequence. Then, Fatou Lemma and the last inequality implies that
\begin{equation}
\int_{C_m}|u_\Omega| dx = 0.
\label{uCm}
\end{equation}
Hence, since $\mathbb{R}^N\backslash \Omega =  \bigcup_{i =1}^{+\infty}C_m$, it follows that
$$
\int_{\mathbb{R}^N\backslash \Omega}|u_\Omega| dx = 0
$$
and then that $u_\Omega = 0$ a.e. in $\mathbb{R}^N\backslash\Omega$.

If $d = 0$, then (\ref{dun}) imply that $\|u_n\|_{\lambda_n} \to 0$ as $n \to +\infty$ and nothing is left to prove.

If $d > 0$, since
$$
d + o_n(1) = \Psi_{\lambda_n}(u_n) \leq \|u_n\|_{\lambda_n},
$$
the same arguments in Lemma \ref{lemmaulambdanaonula} can be used to show that $u^\lambda \neq 0$.

Since $u_\Omega \neq 0$, there exists $t > 0$ such that $tu_\Omega \in \mathcal{N}_\Omega$. Let us prove that $t \in (0,1]$, what is implied by the following claim.

\noindent {\bf Claim.} $\displaystyle \Phi_\Omega'(u_\Omega)u_\Omega \leq 0$.

For $\delta > 0$, let $\varphi_\delta \in C^\infty(\mathbb{R}^N)$ be such that $\varphi_\delta \equiv 1$ in $\Omega_\delta$, $\varphi_\delta \equiv 0$ in $(\Omega_{2\delta})^c$ and $|\nabla \varphi_\delta|_\infty \leq C/\delta$, where by $\Omega_\sigma$ we mean the $\sigma-$neighborhood of $\Omega$, $\sigma > 0$. Let us define 
\begin{equation}
\overline{u_\Omega}(x) = \left\{
\begin{array}{rl}
u_\Omega(x), & \mbox{if $x \in \Omega$,}\\
0, &  \mbox{if $x \in \mathbb{R}^N\backslash \Omega$.}
\end{array} \right.
\label{uoverline}
\end{equation}
Note that $\overline{u_\Omega} \in BV(\mathbb{R}^N)$ and, by the Green Formula for $BV$ functions (see \cite{Buttazzo}[Theorem 10.2.1] for instance),
\begin{equation}
\int_{\Omega_\delta} |D\overline{u_\Omega}| + \int_{\Omega_\delta} |\overline{u_\Omega}|dx = \int_\Omega |Du_\Omega| + \int_\Omega |u_\Omega|dx + \int_{\partial\Omega} |u_\Omega|d\mathcal{H}_{N-1}.
\label{greenformula}
\end{equation}

As in the proof of Lemma \ref{lemmaNehari}, note that $\Psi_{\lambda_n}'(u_n)(\varphi_\delta u_n)$ is well defined and, by using $u_n + t\varphi_\delta u_n$ as test function in (\ref{PScinfinity}) and doing $t \to 0^\pm$, since $(u_n)$ is bounded in $BV(\mathbb{R}^N)$, it is possible to see that

\begin{equation}
\Psi_{\lambda_n}'(u_n)(\varphi_\delta u_n) = o_n(1).
\label{Psion1}
\end{equation}
Then by (\ref{Jlinha}) it follows that
\begin{eqnarray}
\nonumber \Psi_{\lambda_n}'(u_n)(\varphi_\delta u_n) & = & \int_{\mathbb{R}^N}\frac{((Du_n)^a)^2 \varphi_\delta + u_n(Du_n)^a \cdot \nabla \varphi_\delta}{|(Du_n)^a|}dx \\
\nonumber & & + \int_{\mathbb{R}^N}\frac{Du_n}{|Du_n|}\frac{\varphi_\delta(Du_n)^s}{|\varphi_\delta (Du_n)^s|}|\varphi_\delta(D u_n)^s| + \\
\nonumber & & + \int_{\mathbb{R}^N}(1 + \lambda_n V(x)) \mbox{sgn}(u_n)(\varphi_\delta u_n)dx - \int_{\mathbb{R}^N}f(u_n)\varphi_\delta u_n dx\\
\nonumber & = & \int_{\mathbb{R}^N}\varphi_\delta |(Du_n)^a|dx + \int_{\mathbb{R}^N}\frac{u_n(Du_n)^a \cdot \nabla\varphi_\delta}{|(Du_n)^a|}dx +\\
\nonumber & & + \int_{\mathbb{R}^N}\frac{(Du_n)^s}{|(Du_n)^s|}\frac{\varphi_\delta(Du_n)^s}{|\varphi_\delta (Du_n)^s|}|\varphi_\delta(D u_n)^s| + \int_{\mathbb{R}^N}(1 + \lambda_n V(x)) |u_n|\varphi_\delta dx -\\
& & - \int_{\mathbb{R}^N}f(u_n)\varphi_\delta u_n dx.
\label{estimateOmegadelta}
\end{eqnarray}

Since $u_n \to \overline{u_\Omega}$ in $L^q(\Omega_\delta)$ for $1 \leq q < 1^*$, by the lower semicontinuity of $\|\cdot\|_{BV(\Omega_\delta)}$ w.r.t. the $L^q(\Omega_\delta)$ convergence, (\ref{greenformula}), (\ref{Psion1}) and (\ref{estimateOmegadelta}), it follows that
\begin{eqnarray}
\nonumber \|u_\Omega\|_\Omega & \leq & \liminf_{n \to +\infty}\left(\int_{\Omega_\delta}|Du_n| + \int_{\Omega_\delta}|u_n|dx \right)\\
\nonumber & \leq & \liminf_{n \to +\infty}\left( \int_{\mathbb{R}^N}\varphi_\delta |(Du_n)^a|dx + \int_{\mathbb{R}^N}(1 + \lambda_n V(x)) |u_n|\varphi_\delta dx \right.\\
\nonumber & & + \left.\int_{\mathbb{R}^N}\frac{(Du_n)^s}{|(Du_n)^s|}\frac{\varphi_\delta(Du_n)^s}{|\varphi_\delta (Du_n)^s|}|\varphi_\delta(D u_n)^s|\right)\\
\nonumber& \leq & \limsup_{n \to +\infty}\left(\int_{\mathbb{R}^N}f(u_n)\varphi_\delta u_n dx - \int_{\Omega_{2\delta}\backslash\Omega_\delta}\frac{u_n(Du_n)^a \cdot \nabla\varphi_\delta}{|(Du_n)^a|}dx\right)\\
& = & \int_{\Omega_{2\delta}}f(u_\Omega)\varphi_\delta u_\Omega dx,
\label{estimatedeltatozero}
\end{eqnarray}
since 
$$
\int_{\Omega_{2\delta}\backslash\Omega_\delta}\frac{u_n(Du_n)^a \cdot \nabla\varphi_\delta}{|(Du_n)^a|}dx  \leq \frac{C}{\delta}\int_{\Omega_{2\delta}\backslash\Omega_\delta}|u_n|dx = o_n(1),
$$
By doing $\delta \to 0$ in (\ref{estimatedeltatozero}) it follows that $\Phi_\Omega'(u_\Omega)u_\Omega \leq 0$ and the Claim is proved.

Then there exists $t \in (0,1]$ such that $t u_\Omega \in \mathcal{N}_\Omega$. 

Note moreover that
\begin{equation}
d + o_n(1) = \Psi_{\lambda_n}(u_n) + o_n(1) = \Psi_{\lambda_n}(u_n) - \Psi'_{\lambda_n}(u_n)u_n = \int_{\mathbb{R}^N}\left(f(u_n)u_n - F(u_n)\right) dx.
\label{convergenciaf2}
\end{equation}
Applying Fatou Lemma in the last inequality together with Lemma \ref{lemmaincreasing}, we derive that
\begin{eqnarray}
\nonumber d & \geq &  \int_{\Omega}\left(f(u_\Omega)u_\Omega - F(u_\Omega)\right) dx\\
\nonumber & \geq & \int_{\Omega}\left(f(t u_\Omega)t u_\Omega - F(t u_\Omega)\right) dx\\
\nonumber & = & \Phi_\Omega(t u_\Omega) - \Phi'_\Omega(t u_\Omega)t u_\Omega\\
\nonumber & = & \Phi_\Omega(t u_\Omega)\\
& \geq & c_\Omega, 
\label{estimatedcomega}
\end{eqnarray}
what shows that $d \geq c_\Omega$.

Now let us suppose that $d = c_\Omega$. In this case, we can prove that $t = 1$, i.e., that $u_\Omega \in \mathcal{N}_\Omega$. This follows since in this case, from (\ref{estimatedcomega}) and the fact that $d = c_\Omega$, it follows that $t = 1$, $u_\Omega \in \mathcal{N}_\Omega$, $\Phi_\Omega(u_\Omega) = c_\Omega$, and from (\ref{convergenciaf2}),
\begin{equation} \label{EQUAT02}
f(u_n)u_n - F(u_n) \to f(u_\Omega)u_\Omega - F(u_\Omega) \quad \mbox{in} \quad L^{1}(\mathbb{R}^N).
\end{equation}
This limit together with $(f_4)$ imply that
\begin{equation} \label{EQUAT12}
f(u_n)u_n \to f(u_\Omega)u_\Omega \quad \mbox{in $L^1(\mathbb{R}^N)$}
\end{equation}
\begin{equation} \label{EQUAT22}
F(u_n) \to F(u_\Omega) \quad \mbox{in $L^1(\mathbb{R}^N)$}
\end{equation}
and
\begin{equation} \label{EQUAT32}
\|u_n\|_{\lambda_n} \to \|u_\Omega\|_\Omega.
\end{equation}

From (\ref{EQUAT12}), (\ref{EQUAT32}), by taking the $\limsup_{n \to +\infty}$ in $(\ref{PScinfinity})$ and noting that, for all $v \in BV(\Omega)$, if $\overline{v}$ is defined as in (\ref{uoverline}),
$$
\|v\|_\Omega = \|\overline{v}\|_{\lambda_n},
$$
it follows that 
$$
\|v\|_\Omega - \|u_\Omega\|_\Omega \geq \int_\Omega f(u_\Omega)(v - u_\Omega).
$$
Then $u_\Omega$ is a bounded variation solution of (\ref{limitproblem}).
\end{proof}

\subsection{Proof of Theorem \ref{maintheorem}}

Let us consider a sequence $\lambda_n \to +\infty$ as $n \to + \infty$ and, for each $n \in \mathbb{N}$, $u_n := u^{\lambda_n}$ the bounded variation solution of $(P)_{\lambda_n}$ obtained in Section \ref{sectionexistence}, which is such that $\Phi_{\lambda_n}(u_n) = c_{\lambda_n}$.

Note that, for a given $u \in BV(\Omega)$, denoting by $\overline{u}$ its extension by zero outside $\Omega$ (as in (\ref{uoverline})), it follows from Green Formula for $BV$ functions that
\begin{equation}
\int_{\mathbb{R}^N} |D\overline{u}| + \int_{\mathbb{R}^N} |\overline{u}|dx = \int_\Omega |Du| + \int_\Omega |u|dx + \int_{\partial\Omega} |u| d\mathcal{H}_{N-1}.
\label{greenformulaRN}
\end{equation}
Hence, if $u \in BV(\Omega)$, then $\overline{u} \in E_\lambda$ and $\Phi_\Omega(u) = \Psi_\lambda(\overline{u})$ for every $\lambda > 0$. Then, for each $\gamma \in \Gamma_\Omega$, it follows that $\overline{\gamma} \in \Gamma_\lambda$. Based on this fact, it is easy to see that
\begin{equation}
c_\lambda = \inf_{\gamma \in \Gamma_\lambda} \max_{t \in [0,1]} \Psi_\lambda(\gamma(t)) \leq \inf_{\gamma \in \Gamma_\Omega} \max_{t \in [0,1]} \Phi_\Omega(\gamma(t)) = c_\Omega,
\label{estimatedabove}
\end{equation}
for every $\lambda > 0$.

Then it follows that 
$$
(c_{\lambda_n})_{n \in \mathbb{N}} \subset [0,c_\Omega],
$$
which implies that, up to a subsequence, $\Psi_{\lambda_n}(u_{\lambda_n}) \to d \in [0,c_\Omega]$, as $n \to +\infty$. Since $u_n$ satisfies (\ref{PScinfinity}) with $\tau_n = 0$, it follows that $(u_n)$ is in fact a $(PS)_{d,\infty}$ sequence.

Note that by (\ref{estimateclambda}), $d > 0$. On the other hand, by Proposition \ref{(PS)inftycondition}, it holds that
\begin{equation}
d \geq c_\Omega.
\label{estimatedbelow}
\end{equation}

Then, from (\ref{estimatedabove}) and (\ref{estimatedbelow}) it follows that $(u_n)$ is a $(PS)_{c_\Omega,\infty}$-sequence and then, again by Proposition \ref{(PS)inftycondition} there exists $u_\Omega \in BV(\mathbb{R}^N)$ such that, up to a subsequence, $u_n \to u_\Omega$ in $L^q_{loc}(\mathbb{R}^N)$, for $1 \leq q < 1^*$, $u_\Omega \equiv 0$ a.e. in $\mathbb{R}^N\backslash\Omega$ and $u_\Omega$ is a bounded variation solution of (\ref{limitproblem}). Moreover,
$$
\|u_n\|_{\lambda_n} - \|u_\Omega\|_\Omega \to 0, \quad \mbox{as $n \to +\infty$}
$$
and Theorem \ref{maintheorem} is proved.

\vspace{0.3cm}
\noindent {\bf Acknowledgment: } C.O.Alves was partially supported by CNPq/Brazil 304804/2017-7. G. M. Figueiredo is supported by CNPq and FAPDF. M.T.O. Pimenta is supported by FAPESP 2017/01756-2.

\end{document}